\documentclass{amsart}[12 pt]

\usepackage{latexsym}

\usepackage{amsmath,amssymb, amsfonts}
\usepackage{verbatim}
\usepackage{graphicx}
\usepackage{graphics}
\usepackage{geometry}
\usepackage{booktabs}

\newtheorem{theorem}{Theorem}[section]

\newtheorem{corollary}{Corollary}[section]
\newtheorem{lemma}{Lemma}[section]

\numberwithin{equation}{section}

\allowdisplaybreaks

\begin{document}
\markboth{}{}

\title[A note on the sharp $L^p$-Convergence rate of Upcrossings to the Brownian Local-Time] {A Note on the sharp $L^p$-Covergence rate of Upcrossings to the Brownian local time}



\author{Alberto Ohashi}

\address{Departamento de Matem\'atica, Universidade Federal da Para\'iba, 13560-970, Jo\~ao Pessoa - Para\'iba, Brazil}\email{alberto.ohashi@pq.cnpq.br; ohashi@mat.ufpb.br}

\author{Alexandre B. Simas}

\address{Departamento de Matem\'atica, Universidade Federal da Para\'iba, 13560-970, Jo\~ao Pessoa - Para\'iba, Brazil}\email{alexandre@mat.ufpb.br}

\thanks{We would like to thank professor Davar Koshnevisan for very helpful discussions about this article.}
\date{\today}

 \subjclass{}

\begin{center}

\end{center}

\begin{abstract}
In this note, we prove a sharp $L^p$-rate of convergence of the number of upcrossings to the local time of the Brownian motion. In particular, it provides novel $p$-variation estimates ($2 < p < \infty$) for the number of upcrossings of the Brownian motion. Our result complements the fundamental work of Koshnevisan~\cite{kho} who obtains an almost sure exact rate of convergence in the sup norm.
\end{abstract}

\maketitle

\section{Introduction}
Fix $x\in \mathbb{R}$, and let $(\Omega, \mathbb{F}, \mathbb{P}^x)$ be the Wiener space of the canonical Brownian motion starting from $x$, i.e., $\Omega:= C([0,+\infty);\mathbb{R})$ is the space of
continuous functions from $[0,+\infty)$ to $\mathbb{R}$, $\mathbb{F} = (\mathcal{F}_t)_{t\ge 0}$ is the natural filtration generated by the Brownian motion satisfying the usual conditions, and $\mathbb{P}^x$ is the Wiener measure of the Brownian motion $B(\omega,t) = \omega(t)$ with $\mathbb{P}^x\{B(0)=x\}=1$. When $x=0$, we just write $\mathbb{P}$. We endow the Wiener space with the filtration generated by the Brownian motion satisfying the usual conditions.

The occupation measure of $B$ up to the instant $t$ is the measure $\mu_t$ defined by the relation

$$\mu_t(A): = \int_0^t1\!\!1_A(B(s))ds;~A\in \mathcal{B}(\mathbb{R}),$$
where $\mathcal{B}(\mathbb{R})$ is the Borel sigma algebra of $\mathbb{R}$. In a landmark work, L\'evy established that for almost all trajectories of the Brownian motion and for any $t$, the random measure $\mu_t$ has a density and by the classical Trotter Theorem we know that it admits a jointly continuous version $\{\ell^x(t);(x,t)\in \mathbb{R}\times [0,\infty)\} $, the so-called local-time of the Brownian motion.

Several approximation schemes exist in the literature for the Brownian local time. For instance, let $\{\psi_\varepsilon(\cdot); \varepsilon >0\}$ be an approximation to the identity. Then $\lim_{\varepsilon\rightarrow 0}\int_0^t \psi_\varepsilon(B(s) - x)ds = \ell^x(t)~a.s$ uniformly over all $x\in \mathbb{R}$. See e.g~Borodin~\cite{borodin} and Karatzas and Shreve~\cite{karatzas} for further references on this topic. In a different direction, very appealing strong approximation schemes can be constructed from several types of random walks based on the same probability space. See Bass and Koshnevisan~\cite{bass1} for further references. The L\'evy excursion theory provides other approximations schemes by means of the number of upcrossings or of the excursions before a given time.

Sharp rates of almost sure convergence of the number of upcrossings to the Brownian local time in the sup norm and mean squared error are by now well understood. See the fundamental works of Borodin~\cite{borodin}, Koshnevisan~\cite{kho} and Knight~\cite{knight1}. However, much little is known about $L^p(\mathbb{P})$-convergence rates for upcrossings to the Brownian local-time (see e.g~Blandine and Vallois~\cite{vallois} for $L^p$-rates in the sense of regularization). The goal of this short note is to present sharp rates of convergence for the number upcrossings to the Brownian local-time in the $L^p(\mathbb{P})$-sense.

Our main motivation in studying sharp $L^p(\mathbb{P})$-rates of convergence of the number of upcrossings to the Brownian local time lies in It\^o formulas for path-dependent functionals of the Brownian motion. Le\~ao, Ohashi and Simas~\cite{LOS} have recently proved that under suitable $p$-variation regularity~(in the sense of rough path, see e.g~\cite{friz}) of a non-anticipative Brownian functional $F_t:C([0,t]; \mathbb{R})\rightarrow \mathbb{R}$, the process

\begin{equation}\label{pathdec}
F_t(B_t) - \frac{1}{2}\int_0^t\int_{-\infty}^{+\infty} \partial_x F_s(\textbf{t}(B_s,x))d_{(s,x)}\ell^{x}(s),~0\le t\le T,
\end{equation}
is a Brownian semimartingale, where $B_t: = \{B(s); 0\le s\le t\}$  is the Brownian path and $0 < T < \infty$ is a fixed terminal time. Here the $d_{(s,x)}\ell$-integral is the pathwise 2D Young integral (see e.g~\cite{young1,friz,OS}) composed with the Brownian local-time $\{\ell^x(s); 0\le s\le T, x\in\mathbb{R}\}$ and, roughly speaking, a suitable ``space" derivative of $F$ composed with a ``terminal value modification" $\textbf{t}(B_t,x)$ of the Brownian paths (see~\cite{LOS} for further details).

One important step in the proof of~(\ref{pathdec}) is a sharp $L^p(\mathbb{P})$-convergence rate of the number of upcrossings of the embedded random walk introduced by Knight~\cite{knight} to the Brownian local time. In particular, $p$-variation regularity of the number of upcrossings plays a key role on the existence of the semimartingale decomposition~(\ref{pathdec}) and it is an almost immediate corollary of the main result of this note.

\section{Preliminaries}
At first, let us recall the F. Knight~\cite{knight} construction of an $2^{-k}\mathbb{Z}$-valued simple symmetric random walk using a single Brownian motion.
For a fixed positive integer $k$, we define $T^{k}_0 := 0$ a.s. and

$$
T^{k}_n := \inf\{T^{k}_{n-1}< t <\infty;  |B(t) - B(T^{k}_{n-1})| = 2^{-k}\}, \quad n \ge 1.
$$
Then the discrete-time process $R^k:=\{B(T^k_n); n\ge 0 \}$ is a simple symmetric random walk. For simplicity of exposition, we are going to imbed $R^k$ into a continuous-time process. We define $A^k$ as follows

\[
A^{k} (t) := \sum_{n=1}^{\infty}2^{-k}\eta^{k}_n1\!\!1_{\{T^{k}_n\leq t \}};~t\ge0,
\]
where

$$
\eta^{k}_n:=\left\{
\begin{array}{rl}
1; & \hbox{if} \ B (T^{k}_n) - B (T^{k}_{n-1}) = 2^{-k} ~ ~ \mbox{and} ~ ~ T^{k}_n < \infty \\
-1;& \hbox{if} \ B (T^{k}_n) - B (T^{k}_{n-1}) = -2^{-k} ~ ~ \mbox{and}~ ~ T^{k}_n < \infty \\
0; & \hbox{if} \ T^{k}_n = \infty,
\end{array}
\right.
$$
for $k,n\ge 1$. Then, for each $k\ge 1$, $A^k$ is a bounded variation martingale w.r.t its natural filtration $\mathbb{F}^k$. See~e.g~\cite{LEAO_OHASHI2013} for details. In the sequel, for a given $x\in \mathbb{R}$, let $j_k(x)$ be the unique integer such that $(j_k(x)-1)2^{-k} < x \le j_k(x)2^{-k}$. Let us define

$$u(j_k(x)2^{-k},k,t):= \#\ \Big\{n \in \{0, \ldots, N^k(t)-1\}; A^k(T^k_{n}) =(j_k(x)-1)2^{-k}, A^k(T^k_{n+1}) =j_k(x)2^{-k}\Big\};$$
for~$x\in \mathbb{R}, k\ge 1, 0\le t < \infty.$ Here, $N^k(t): = \max \{n; T^k_n \le t\}$ is the length of the embedded random walk until time $t$. By the very definition, $u(j_k(x)2^{-k},k,t) :=$ number of upcrossings of~$A^k$~from~$(j_k(x)-1)2^{-k}$~to~$j_k(x)2^{-k}$~before time~$t$. In the sequel, we are going to denote

$$U^k(t,x):=22^{-k}u(j_k(x)2^{-k},k,t); (t,x)\in [0,+\infty)\times \mathbb{R}.$$
One fundamental result due to Koshnevisan~(see Th. 1.4 and Remark~1.7.1 in~\cite{kho}) provides the exact rate of almost sure convergence of the number of upcrossings to the Brownian local time as follows:

\begin{theorem}\label{koshrr}
For every finite positive constant $M>0$,
\begin{equation}\label{koshr}
\lim_{k\rightarrow \infty}\sup_{0\le t\le M}\Bigg|\sup_{x\in \mathbb{R}} \frac{|U^k(t,x) - \ell^x(t)|}{\sqrt{2^{-k}log(2^k)}} - 2\sqrt{\ell^*(t)}\Bigg| = 0~\text{almost surely}.
\end{equation}
where $\ell^{*}(t):=\sup_{x\in \mathbb{R}}\ell^x(t)$.
\end{theorem}
In this article, we are going to show a counterpart of Theorem~\ref{koshrr} in the sense of $L^p(\mathbb{P})$ for $1\le p <\infty$. More precisely, our main result reads as follows.
\begin{theorem}\label{lprate}
Let $ T$ be a finite positive constant. For each $\eta >0$, there exists a finite universal positive constant $C(\eta)$ such that

$$\sup_{k\ge 1}\mathbb{E}\sup_{0\le t\le T}\sup_{x\in \mathbb{R}}\frac{ |U^k(t,x) - \ell^x(t)|^{2+\eta}}{ (2^{-k}log(2^k))^{\frac{2+\eta}{2}} } \le C(\eta)T^{\frac{1}{2} + \frac{\eta}{4}}.$$
\end{theorem}

In the sequel, we denote $I_m := [-2^m,2^m]\subset \mathbb{R}$ where $m\ge 1$ is a given positive integer. We recall the notion of $q$-variation of a real-valued function $f:I_m\rightarrow \mathbb{R}$

$$\|f\|^q_{I_m;q}:= \sup_{\Pi}\sum_{x_i\in \Pi}|f(x_i)-f(x_{i-1})|^q;~1\le q< \infty,$$
where $\sup$ is taken over all partitions of the compact set $I_m$ (see~e.g~\cite{friz}). An important result due to Feng and Zao~\cite{feng} states that the Brownian local time has finite $(2+\delta)$-variation in the sense that $\mathbb{E}\sup_{0\le t\le T}\|\ell(t) \|^{2+\delta}_{I_m;2+\delta}< \infty$ for every $\delta>0$. In this article, we show an upper bound for the number upcrossings as follows. An almost immediate corollary of Theorem~\ref{lprate} is the following result:

\begin{corollary}\label{variationr}
For any $\delta>0$ and $m\ge1$, there exists a finite positive constant $C(\delta,m)$ (which only depends on $\delta$ and $m$) such that

$$\sup_{k\ge 1}\mathbb{E}\sup_{0\le t\le T}\|U^k(t)\|^{2+\delta}_{I_m; 2+\delta}\le \mathbb{E}\sup_{0\le t\le T}\|\ell(t)\|^{2+\delta}_{I_m;2+\delta} + C(\delta,m)T^{\frac{1}{2}+\frac{\delta}{4}}.$$
\end{corollary}
\section{Proof of Theorem~\ref{lprate}}
It is not difficult to see that Theorem~\ref{koshrr} will play a key role in the proof  of Theorem~\ref{koshr}. However, the argument given by Koshnevisan in the proof of~(\ref{koshr}) is fully probabilistic in the sense that it essentially relies on purely Borel-Cantelli's-type arguments with none $L^p$ estimates at hand. So we need to adopt a rather different strategy. Thanks to the deep Burh\"{o}lder's ideas~\cite{burk} on moderate functions, we shall construct an argument towards the proof of Theorem \ref{lprate}. The strategy is the obtention of a so-called good-lambda inequality to get the desired $L^p$-rate of convergence. See Jacka~\cite{jacka} and Bass~\cite{bass} for further details.

The starting point of our analysis is the study of the scaling behavior of the following adapted process

$$\sup_{0\le s\le t} \sup_{x\in \mathbb{R}} \frac{ |U^k(s,x) - \ell^x(s)|^2} { 2^{-k}log(2^k)};~0\le t <\infty.$$
The left-continuous version is given by
$$J^k(t):=\sup_{0\le s < t}\sup_{x\in \mathbb{R}} \frac{ |U^k(s,x) - \ell^x(s)|^2} { 2^{-k}log(2^k)};~0\le t< \infty.$$
Since we are only interested on the bounded set $[0,T]$, we are going to stop $J^k$ as follows

$$F^k(t) : = J^k(t\wedge \bar{T}); 0\le t < \infty$$
where $ T < \bar{T}< \infty$.

\begin{lemma}\label{unifscal}
The following convergence holds
$$\lim_{b\rightarrow \infty} \sup_{k\ge 1}\sup_{x\in \mathbb{R}, \lambda >0} \mathbb{P}^x\Big\{ F^k(\lambda^2) > b\lambda  \Big  \}=0.$$
\end{lemma}
\begin{proof}
Let us fix $k\ge 1$. By the very definition,

\begin{eqnarray}
\nonumber\sup_{y\in\mathbb{R}, \lambda > 0}\mathbb{P}^y\{F^k(\lambda^2) > b\lambda \} &\le& \mathop{\sup_{y\in \mathbb{R}}}_{0 <  \lambda^2 < \bar{T}}\mathbb{P}^y\{F^k(\lambda^2) > b\lambda \}+ \mathop{\sup_{y\in \mathbb{R}}}_{\lambda^2\ge \bar{T}}\mathbb{P}^y\{F^k(\bar{T}) > b\lambda \}\\
\nonumber& &\\
\label{fff}&\le& \mathop{\sup_{y\in \mathbb{R}}}_{0 < \lambda^2 < \bar{T}}\mathbb{P}^y\{F^k(\lambda^2) > b\lambda \} + \sup_{y\in \mathbb{R}}\mathbb{P}^y\{F^k(\bar{T}) > b\bar{T}^{1/2} \}.
\end{eqnarray}
By the very definition, if $G$ is a Borel subset of $\Omega$ and $y\in \mathbb{R}$ then $\mathbb{P}^y(G)=\mathbb{P}(G-y)$ where $G-y: =\{\omega\in \Omega; \omega(\cdot) + y \in G\}$. Now the $\mathbb{P}^y $ law of $F^k(\bar{T})$ does not dependent on $y\in \mathbb{R}$ because of the sup over all the initial conditions in $\mathbb{R}$. Then, the almost sure convergence~(\ref{koshr}) and the fact that $\ell^*(\bar{T}) < \infty~a.s$ yield

$$\sup_{y\in \mathbb{R}}\mathbb{P}^y\{F^k(\bar{T}) > b\bar{T}^{1/2} \} \le \mathbb{P}\{\sup_{r\ge 1}F^r(\bar{T}) > b \bar{T}^{1/2}\}\rightarrow 0$$
as $b\rightarrow \infty$ uniformly in $k\ge 1$. It remains to estimate the first term in~(\ref{fff}). We notice that for a given $\lambda >0$, the map $s\mapsto s\lambda$ is a bijection from $[0,\lambda]$ onto $[0,\lambda^2]$ and hence

\small
\begin{eqnarray}\label{bij}
\sup_{0< \lambda^2 < \bar{T}}\sup_{y\in \mathbb{R}}\mathbb{P}^y\{F^k(\lambda^2) > b\lambda\}&=& \sup_{0< \lambda^2 < \bar{T}}\sup_{y\in \mathbb{R}}\mathbb{P}^y
\Bigg\{\sup_{0\le s< \lambda}\sup_{x\in \mathbb{R}}\frac{ |\frac{1}{\sqrt{\lambda}}U^k(s\lambda,\sqrt{\lambda}x) - \frac{1}{\sqrt{\lambda}}\ell^{\sqrt{\lambda}x}(s\lambda)|^2}{2^{-k}log(2^k) } > b  \Bigg  \}.
\end{eqnarray}
\normalsize
Moreover, the Brownian motion scaling invariance yields

\small
\begin{equation}\label{gd}
\mathbb{P}^y\Bigg\{\sup_{0\le s< \lambda}\sup_{x\in \mathbb{R}}\frac{ |U^k(s,x) - \ell^x(s)|^2}{ 2^{-k}log(2^k)}  > b  \Bigg  \} = \mathbb{P}^y\Bigg\{\sup_{0\le s <  \lambda}\sup_{ x\in \mathbb{R}}\frac{ |\frac{1}{\sqrt{\lambda}}U^k(s\lambda,\sqrt{\lambda}x) - \frac{1}{\sqrt{\lambda}}\ell^{\sqrt{\lambda}x}(s\lambda)|^2}{ 2^{-k}log(2^k) } > b  \Bigg  \}.
\end{equation}
\normalsize
Summing up identities~(\ref{bij}) and (\ref{gd}), the fact that the $\mathbb{P}^y$ law of  $\sup_{0\le s< \lambda}\sup_{x\in \mathbb{R}}\frac{ |U^k(s,x) - \ell^x(s)|^2}{ 2^{-k}log(2^k)}$ does not depend on $y\in \mathbb{R}$ and using~(\ref{koshr}), we do have
\small
\begin{eqnarray*}
\nonumber\mathop{\sup_{y\in \mathbb{R}}}_{0< \lambda^2 < \bar{T}} \mathbb{P}^y\{F^k(\lambda^2) > b\lambda\}&\le& \mathbb{P}\Bigg\{\sup_{0\le s\le \bar{T}^{1/2}} \sup_{x\in \mathbb{R}} \frac{|U^k(s,x) - \ell^x(s)|^2}{2^{-k}log(2^k)} > b  \Bigg\}\\
\nonumber& &\\
&\le &\mathbb{P}\Bigg\{\sup_{r\ge 1}\sup_{0\le s\le \bar{T}^{1/2}} \sup_{x\in \mathbb{R}} \frac{|U^r(s,x) - \ell^x(s)|^2}{2^{-r}log(2^r)} > b  \Bigg\}\rightarrow 0
\end{eqnarray*}
\normalsize
as $b\rightarrow \infty$ uniformly in $k\ge 1$. This concludes the proof.
\end{proof}

In the sequel, let $\theta_t:\Omega\rightarrow \Omega$ be the shift operator defined by $\theta_s \omega:= \omega(\cdot+s);~\omega\in \Omega$. For a given adapted process $\{H(t); t\ge 0\}$, we recall that $H(\cdot)\circ \theta_s(\omega):=H(\omega, \cdot + s)$ for each $\omega\in \Omega$ and $s\ge 0$.



\begin{lemma}\label{sub}
For each $k\ge 1$, the functional $F^k$ is an $\mathbb{F}$-adapted non-decreasing functional with left-continuous paths which satisfies the following sub-additivity relation: For every $0\le s\le t < \infty$, we have

\begin{equation}\label{subad}
F^k(t)-F^k(s) \le F^k(t-s)\circ \theta(s)~a.s
\end{equation}

\end{lemma}
\begin{proof}
The fact that $F^k$ is adapted with left-continuous and non-decreasing paths is obvious. If $\bar{T} \le s < t$, then~(\ref{subad}) trivially holds. Now, if $0\le s< t \le \bar{T} $, we clearly have

$$\sup_{0\le r < t}\sup_{x\in \mathbb{R}} \frac{ |U^k(r,x) - \ell^x(r)|^2} { 2^{-k}log(2^k)}\le \sup_{0\le r < s}\sup_{x\in \mathbb{R}} \frac{ |U^k(r,x) - \ell^x(r)|^2} { 2^{-k}log(2^k)} + \sup_{s\le r < t}\sup_{x\in \mathbb{R}} \frac{ |U^k(r,x) - \ell^x(r)|^2} { 2^{-k}log(2^k)}.
$$
Observe that the functional $F^k$ only depends on the time variable, and does not depend on the space variable. Therefore, from the definition of the shift operator, we obtain:
$$\sup_{s\le r < t}\sup_{x\in \mathbb{R}} \frac{ |U^k(r,x) - \ell^x(r)|^2} { 2^{-k}log(2^k)} =\sup_{0\le r < t-s}\sup_{x\in \mathbb{R}} \frac{ |U^k(r,x) - \ell^x(r)|^2} { 2^{-k}log(2^k)}\circ\theta_s.$$
This concludes the proof.
\end{proof}
We are now able to prove Theorem~\ref{lprate}. The idea is to find a good-lambda inequality~(see~e.g~\cite{jacka},~\cite{bass}) for our functional $F^k$. For this purpose, we fix $k\ge 1$, $\eta>0$ and for a given $\lambda >0$, let us define $J_{\lambda}: =\inf\{t\ge 0; F^k(t) > \lambda\}$. We also fix $\beta >1$ and $\delta>0$. Since $F^k$ is left-continuous, then $F^k(J_{\lambda}-) = F^k(J_{\lambda})\le \lambda$~a.s. By using the sub-additive property of $F^k$ given in Lemma~\ref{sub}, the strong Markov property of the Brownian motion, the non-decreasing and left-continuous paths of $F^k$, we shall find a good-lambda inequality as follows

\small
\begin{eqnarray*}
\mathbb{P}\{F^k(\bar{T}) > \beta\lambda, (\bar{T})^{1/2}\le \delta \lambda \}&\le& \mathbb{P}\{F^k(\bar{T}) - F^k(J_{\lambda}) > (\beta-1)\lambda, (\bar{T})\le \delta^2 \lambda^2 \}\\
& &\\
&\le & \mathbb{P}\{F^k(J_{\lambda} + \delta^2\lambda^2) - F^k(J_{\lambda})\ge (\beta-1)\lambda, J_{\lambda} < \bar{T}\}\\
& &\\
&\le& \mathbb{P}\{F^k(\delta^2\lambda^2)\circ \theta_{J_{\lambda}}\ge (\beta-1)\lambda, J_{\lambda} < \bar{T}\}\\
& &\\
&=& \int_{\{J_{\lambda}< \bar{T}\}}\mathbb{P}\big[F^k(\delta^2\lambda^2)\circ \theta_{J_{\lambda}}\ge (\beta-1)\lambda|\mathcal{F}_{J_{\lambda}}\big]d\mathbb{P}\\
& &\\
&=& \int_{\{J_{\lambda} < \bar{T}\}}\mathbb{P}^{B(J_{\lambda})}\big[F^k(\delta^2\lambda^2)\ge (\beta-1)\lambda\big]d\mathbb{P}\\
& &\\
&\le& \sup_{x\in\mathbb{R}}\mathbb{P}^x\{F^k(\delta^2\lambda^2)\ge(\beta-1)\lambda   \}\mathbb{P}\{J_{\lambda} < \bar{T}\}\\
& &\\
&\le& \mathop{\sup_{x\in \mathbb{R}}}_{\eta >0} \mathbb{P}^x\Big\{F^k(\eta^2) > \frac{\beta-1}{2\delta}\eta   \Big\}\mathbb{P}\big\{F^k(\bar{T})>\lambda \big\}.
\end{eqnarray*}
\normalsize
Now, from Lemma~\ref{unifscal}, we shall take $\delta$ small enough in such way that
$$\mathop{\sup_{x\in \mathbb{R}}}_{\eta >0} \mathbb{P}^x\Big\{F^k(\eta^2) > \frac{\beta-1}{2\delta}\eta   \Big\}$$
is small uniformly in $k\ge 1$. Then, by applying Lemma 7.1 given in Burkh\"{o}lder~\cite{burk} on the moderate function $x\mapsto x^{1+\frac{\eta}{2}}$ ($x\ge 0$), we get an universal constant which only depends on $\eta$ such that

\begin{equation}\label{l}
\mathbb{E}|F^k(\bar{T})|^{1+\frac{\eta}{2}}\le C(\eta)(\bar{T})^{\frac{1}{2}+\frac{\eta}{4}}~\forall k\ge 1.
\end{equation}
\noindent Since $\bar{T}> T$ is arbitrary, then~(\ref{l}) concludes the proof of Theorem~\ref{lprate}.

\subsection{Proof of Corollary~\ref{variationr}}
Let us now give the proof of Corollary~\ref{variationr}. In the sequel, we fix $\delta >0, m\ge 1$ and for a given partition $\Pi=\{x_i\}_{i=0}^N$ of the interval $I_m$, let us define the following subset $\Lambda(\Pi,k): = \{x_i\in \Pi; (j_k(x_i)-j_k(x_{i-1}))2^{-k} >0 \}$. We notice that $\#\ \Lambda(\Pi,k)\le 22^{k+m}$ for every partition $\Pi$ of $I_m$. We readily see that

\begin{equation}\label{es1}
\sum_{x_i\in \Pi}|U^{k}(t,x_i)  - U^{k}(t,x_{i-1})|^{2+\delta}\le \sum_{x_i\in \Lambda(\Pi,k)}|U^{k}(t,x_i)  - U^{k}(t,x_{i-1})|^{2+\delta},
\end{equation}
for every partition $\Pi$ of $I_m$. By writing $|U^{k}(t,x_i)  - U^{k}(t,x_{i-1})| = |U^{k}(t,x_i) - \ell^{x_i}(t) + \ell^{x_i}(t) - \ell^{x_{i-1}}(t) + \ell^{x_{i-1}}(t) - U^{k}(t,x_{i-1})|;~x_i\in \Lambda(\Pi,k)$ and applying the standard inequality $|\alpha -\beta|^{2+\delta}\le 2^{1+\delta}\{|\alpha|^{2+\delta} + |\beta|^{2+\delta}\};~\alpha,\beta\in \mathbb{R}$, we get from~(\ref{es1})

\begin{eqnarray}
\nonumber \mathbb{E}\sup_{0\le t\le T}\|U^k(t)\|^{2+\delta}_{I_m;2+\delta} &\le& C\mathbb{E}\sup_{0\le t\le T}\|\ell^x(t)\|^{2+\delta}_{I_m;2+\delta} + \nonumber C\mathbb{E}\sup_{0\le t\le T}\sup_{\Pi}\sum_{x_i\in \Lambda(\Pi,k)}|U^k(t,x_i) - \ell^{x_i}(t)|^{2+\delta}\\
\nonumber& &\\
\label{l2}&+& C \mathbb{E}\sup_{0\le t\le T}\sup_{\Pi}\sum_{x_i\in \Lambda(\Pi,k)}|U^k(t,x_{i-1}) - \ell^{x_{i-1}}(t)|^{2+\delta},
\end{eqnarray}
for a constant $C$ which only depends on $\delta>0$. An inspection in the proof of Lemma 2.1 in Feng and Zao~\cite{feng} yields $\mathbb{E}\sup_{0\le t\le T}\|\ell^x(t)\|^{2+\delta}_{I_m;2+\delta}<\infty$. Now,

$$\sup_{0\le t\le T}\sup_{\Pi}\sum_{x_i\in \Lambda(\Pi,k)}|U^k(t,x_i) - \ell^{x_i}(t)|^{2+\delta}\le 22^{k+m} \mathop{\sup_{x\in I_m}}_{0\le t\le T}|U^k(t,x) - \ell^{x}(t)|^{2+\delta}~a.s~k\ge 1$$
and hence from Theorem~\ref{lprate}, we have

\begin{eqnarray*}
\mathbb{E}\sup_{0\le t\le T}\sup_{\Pi}\sum_{x_i\in \Lambda(\Pi,k)}|U^k(t,x_i) - \ell^{x_i}(t)|^{2+\delta}&\le& 22^{k+m}C(\delta)T^{\frac{1}{2}+\frac{\delta}{4}}(2^{-k}klog(2))^{1+\frac{\delta}{2}}\\
& &\\
&\le& 2^{m+1}log(2)^{1+\frac{\delta}{2}} T^{\frac{1}{2}+\frac{\delta}{4}}\sup_{r\ge 1} 2^{\frac{-r\delta}{2}}r^{1+\frac{\delta}{2}}\\
& &\\
&\le& C(\delta,m)T^{\frac{1}{2}+\frac{\delta}{4}} < \infty,
\end{eqnarray*}
for some constant  $c(\delta,m)$ which only depends on $m,\delta$. The other term in~(\ref{l2}) can be treated similarly. This concludes the proof of Corollary~\ref{variationr}.

\end{document}